\documentclass[11pt,a4paper]{article}
\usepackage{amsmath, amssymb, amscd, amsthm, amsfonts}
\usepackage{graphicx}
\usepackage[numbers]{natbib}
\usepackage{hyperref}
\usepackage{wasysym}
\usepackage{comment}
\usepackage{tikz}
\usepackage{subcaption}
\usepackage{graphicx}
\usepackage[title]{appendix}

\newtheoremstyle{myplain}
  {3pt}{3pt}
  {\normalfont}
  {}
  {\bfseries}
  {.}
  {0.5em}
  {}
\numberwithin{equation}{section}

\theoremstyle{myplain}
\newtheorem{theorem}{Theorem}[section]
\newtheorem{lemma}[theorem]{Lemma}
\newtheorem{claim}[theorem]{Claim}
\newtheorem{case}{Case}

\newtheorem{problem}[theorem]{Problem}

\newtheorem{corollary}[theorem]{Corollary}
\newtheorem{definition}[theorem]{Definition}
\newtheorem{remark}[theorem]{Remark}

\begin{document}

\title{Tight bounds for judicious $3$-partitions of graphs}
\author{Peiru Kuang \footnote{School of Mathematical Sciences, Shanghai Jiao Tong University, Shanghai 200240, China. Email: peiru\_k@sjtu.edu.cn} 
\and 
Yan Wang \footnote{School of Mathematical Sciences, Shanghai Jiao Tong University, Shanghai 200240, China. Supported by National Key R\&D Program of China under Grant No. 2022YFA1006400, National Natural Science Foundation of China (Nos. 12571376) and Shanghai Municipal Education Commission (No. 2024AIYB003). Email: yan.w@sjtu.edu.cn (corresponding author).}}

\date{}

\maketitle
\begin{abstract}
In this paper, we show that every graph with $m$ edges admits a 3-partition such that 
\[
\max_{1 \leq i \leq 3} e(V_i) \leq \frac{m}{9} + \frac{1}{9}h(m)
\quad \text{and} \quad
e(V_1, V_2, V_3) \geq \frac{2}{3}m + \frac{1}{3}h(m),
\]
where $h(m) = \sqrt{2m + 1/4} - 1/2$. This answers a problem of Bollobás and Scott affirmatively. We also solve several related problems of Bollobás and Scott. All of our results are tight.
\end{abstract}

\section{Introduction}\label{section-introduction}
Let $G$ be a simple graph. Let $k\geq 2$ be an integer. A \textit{$k$-partition} of $G$, denoted by $(V_1,V_2,\ldots,V_k)$, is a partition of $V(G)$ into $k$ pairwise disjoint non-empty sets $V_i$, $i\in [k]$. For any $k$-partition $(V_1, V_2, \ldots, V_k)$, we define the number of \textit{crossing edges} (edges with end vertices in distinct parts) as $e_G(V_1, V_2, \ldots, V_k) = \sum_{1 \leq i < j \leq k} e(V_i, V_j)$. When clear from the text, the subscript $G$ may be omitted. 

Partition problems in combinatorics and computer science typically seek a partition of a combinatorial structure that optimizes certain parameters. A classical example is the \emph{Max-Cut problem}, which asks for a bipartition (i.e. a $2$-partition) maximizing the number of crossing edges; see~\cite{FMM2002, GM1995, K2009}. Another interesting type of problems seek a partition that optimizes multiple quantities simultaneously. These are known as \textit{judicious partition problems}, see \cite{BS2002, BS2002problems, BS2004, scott2005} for more references.

\subsection{Large cuts}
For a graph $G$ and an integer $k \geq 2$, let $f_k(G)$ denote the maximum number of edges in a $k$-partite subgraph of $G$. Define $f_k(m):=\min \{f_k(G):e(G)=m\}$. 

The classical \textit{Max-Cut problem} seeks to compute $f_2(G)$ and is NP-complete. It is easy to see that $f_2(G)\geq m/2$, by considering the expected number of crossing edges in a random partition. Considerable effort has been devoted to establishing lower bounds of $f_2(G)$. One of the most notable results, by Edwards~\cite{Edwards1973, Edwards1975}, states that
\begin{equation} \label{max-2-cut}
 f_2(G) \geq \frac{m}{2} + \frac{1}{4}h(m),   
\end{equation}
where $h(m)=\sqrt{2m+1/4}-1/2$. This bound is sharp, as $K_{2r+1}$ are extremal graphs for every positive integer $r$.
In \cite{BS2002}, Bollob{\'a}s and Scott extended \eqref{max-2-cut} and showed that 
\begin{equation} \label{max-k-cut}
 f_k(G)\ge \frac{k-1}{k}m+\frac{k-1}{2k}h(m)-\frac{(k-2)^2}{8k}.   
\end{equation}
Note that \eqref{max-k-cut} is best possible for complete graphs when $k$ is even \footnote{We note that the constant term $+\frac{k^2 - 2k + 2}{8k}$ appearing in \cite{BS2002} is incorrect and it should be $-\frac{k^2 - 2k + 2}{8k}$.}. In this paper, we establish the following tight result for \(k = 3\).
\begin{theorem} \label{maxcut k=3}
Let $G$ be a graph with $m$ edges. Then
  \begin{equation} \label{m3c}
        f_3(G)\ge \frac{2}{3}m+\frac{1}{3}h(m),
  \end{equation}
and equality holds if and only if $G=K_n$ where $n$ is not divisible by $3$. 
\end{theorem}

\subsection{Judicious partitions}
Bollob{\'a}s and Scott initiated the study of \textit{judicious partition problems}. The famous \textit{bottleneck bipartition problem} by Entringer~\cite{SS1994} seeks a bipartition that minimizes $\max\{e(V_1),e(V_2)\}$. In \cite{BS1999}, Bollob{\'a}s and Scott proved that every graph admits a $k$-partition in which each part spans relatively few edges, i.e.,
\begin{equation} \label{min-eVi}
\max_{1\leq i \leq k} e(V_i) \leq \frac{m}{k^2} + \frac{k-1}{2k^2} h(m),
\end{equation}
which is tight for \( K_{kr+1} \) for every positive integer \( r \). In fact, in the same paper~\cite{BS1999}, they established a stronger result, demonstrating that every graph admits a $2$-partition that satisfies both inequalities \eqref{max-2-cut} and \eqref{min-eVi} with $k=2$,
and the complete graphs $K_{2r+1}$ are the only extremal graphs (modulo isolated vertices). Subsequently, Bollob\'{a}s and Scott~\cite{BS2002} also asked the following more general problem.
\begin{problem}[Bollob\'{a}s and Scott~\cite{BS2002}] \label{problem}
Does any graph $G$ of size $m$ have a $k$-partition satisfying both 
\[
\max\limits_{1\le i \le k}\{e({{V}_{i}})\} \le \frac{m}{k^2}+\frac{k-1}{2k^2}h(m)
\quad \text{and} \quad
e({{V}_{1}},\ldots,{{V}_{k}})\ge \frac{k-1}{k}m+\frac{k-1}{2k}h(m)-\frac{(k-2)^2}{8k}?
\]
\end{problem}

A weaker version of Problem \ref{problem} was obtained by \cite{XY2009}, which showed that \( G \) admits a $k$-partition $\left( V_1, \ldots, V_k \right)$ such that
$\max_{1\le i \le k}\{e({{V}_{i}})\} \leq m/k^2 + (k-1)h(m)/(2k^2)$ and
%\quad \text{and} \quad
$e(V_1, \ldots, V_k) \geq (k-1)m/k+ h(m)/(2k)$.
For further improvements, we refer the reader to \cite{FHZ2014, GNYZ2025, LX2016, XY2011}.
 
The main result of this paper is the following theorem which gives an affirmative answer to Problem \ref{problem} when $k=3$.
\begin{theorem} \label{main result}
Every graph $G$ with $m$ edges admits a $3$-partition $({{V}_{1}},V_2,{{V}_{3}})$ such that
\begin{enumerate}
    \item [(i)] $\max \limits_{1\le i \le 3}\{e({{V}_{i}})\} \le \frac{m}{9}+\frac{1}{9}h(m)$, and
    \item [(ii)] $e({{V}_{1}},{{V}_{2}},{{V}_{3}})\ge \frac{2}{3}m+\frac{1}{3}h(m)$.
\end{enumerate}
\end{theorem}
\begin{remark}
Theorem \ref{main result} is tight since $K_{3r+1}\cup p K_1$ are extremal graphs for any positive integers $r$ and $p$.
\end{remark}
A crucial ingredient in the proof of Theorem \ref{main result} is the following stability result for judicious 2-partitions: Either $G= K_n$ or one can obtain a stronger judicious partition result than \cite{BS1999}.
\begin{theorem} \label{stability}
For any graph $G$ with $m$ edges, we have either
\begin{itemize}
    \item[(i)] $G$ is a complete graph of odd order, or
    \item[(ii)] $f_2(G)\geq m/2+h(m)/4+1/4$, and there exists a 2-partition \( V(G) = V_1 \cup V_2 \) such that 
     \[\max\limits_{1\le i \le 2}e(V_i) \leq \frac{m}{4} + \frac{1}{8}h(m)\quad \text{and} \quad
    e(V_1, V_2) \geq \frac{m}{2} + \frac{1}{4}h(m)+\frac{1}{4}.\]
\end{itemize}
\end{theorem}
In addition, for every $k\ge3$, we solve Problem \ref{problem} for graphs with not many edges.
\begin{theorem} \label{small edges} Let $k\geq 3$ be an integer. Every graph $G$ with $m<2k^2$ edges admits a $k$-partition $({{V}_{1}},\ldots,{{V}_{k}})$ such that
\begin{enumerate}
    \item [(i)] $\max\limits_{1\le i \le k}\{e({{V}_{i}})\} \le \frac{m}{k^2}+\frac{k-1}{2k^2}h(m)$, and
    \item [(ii)] $e({{V}_{1}},\ldots,{{V}_{k}})\ge \frac{k-1}{k}m+\frac{k-1}{2k}h(m)-\frac{(k-2)^2}{8k}$.
\end{enumerate}
\end{theorem}
In fact, Theorem~\ref{main result} is closely related to another problem of Bollob{\'a}s and Scott~\cite{BS2002problems}.
\begin{problem}[Bollob\'{a}s and Scott~\cite{BS2002problems}] \label{related}
For \( k \geq 3 \), what is the largest constant \( c(k) \) such that for every graph \( G \) with \( m \) edges, there exists a partition \( V(G) = \bigcup_{i=1}^k V_i \) satisfying for all \( 1 \leq i \leq k \)
\[
\binom{k+1}{2} e(V_i) + c(k) \sum_{\substack{j=1 \\ j \neq i}}^k e(V_j) \leq m?
\]
\end{problem}
They~\cite{BS2002problems} note that $c(k) = k/2 $ (if true) would be best possible. We show it is true when $k=3$.
\begin{corollary}
Every graph \( G \) with \( m \) edges admits a $3$-partition \( V(G) = \bigcup_{i=1}^3 V_i \) satisfying
\[
\binom{4}{2} e(V_i) + \frac{3}{2} \sum_{\substack{j=1 \\ j \neq i}}^3 e(V_j) \leq m
\]
for all \( 1 \leq i \leq 3 \).
\end{corollary}

\begin{proof}
By Theorem \ref{main result}, let \( (V_1, V_2, V_3 )\) be a $3$-partition of \( V(G) \) such that
\[
e(V_1,V_2, V_3) \geq \frac{2}{3} m + \frac{1}{3}h(m)
\quad
\text{and}
\quad
\max\limits_{1\leq i \leq 3}e(V_i) \leq \frac{m}{9} + \frac{1}{9}h(m).\]
Then
\[
\begin{aligned}
\binom{4}{2} e(V_i) + \frac{3}{2} \sum_{j \neq i} e(V_j) 
&= \frac{3}{2} \left( 3e(V_i) + \sum_{j=1}^{3} e(V_j) \right) 
= \frac{3}{2} \left( 3e(V_i) + m - e(V_1, V_2, V_3) \right) \\
&\leq \frac{9}{2}\left(\frac{m}{9}+\frac{1}{9}h(m)\right)+\frac{3}{2}\left(m-\frac{2}{3}m-\frac{1}{3}h(m)\right) = m.
\end{aligned}
\]
The complete graphs $K_{3r+1}$ are extremal graphs and thus $c(3)=3/2$ is best possible.
\end{proof}

This paper is organized as follows. Section 2 is devoted to establishing results on large cuts (Theorem \ref{maxcut k=3}).
In Section 3, we provide a constructive proof of Theorem \ref{small edges}. We develop some key lemmas in Section 4, which will be used in the proof of our main result.
Sections 5 and 6 present our main results: we first establish a stability result of judicious $2$-partitions (Theorem \ref{stability}), then prove our main result (Theorem \ref{main result}).

\section{Large cuts}
In this section, we prove Theorem \ref{maxcut k=3}. The main idea of the proof is to map a partition of a simple graph to a partition of the corresponding weighted complete graph, while both partitions of them have the same number (or total weight) of crossing edges. First, we compute the maximum cut of complete graphs. 
\begin{definition}
A $k$-partition $(V_1,\ldots,V_k)$ of $G$ is called \textit{balanced} if $\left||V_i|-|V_j|\right|\leq 1$ for all $i,j\in [k]$.
\end{definition}
\begin{lemma} \label{Kn-cut}
Let $n$, $r$, $k$ and $s\in [k-1]$ be integers such that $n=kr+s$, and let $m=e(K_{n})$. Then
  \[
  f_k(K_{n})= \frac{k-1}{k}m+\frac{k-1}{2k}h(m)+\frac{k-1}{2k}-\frac{s(k-s)}{2k}.
  \]
\end{lemma}

\begin{proof}
Let $(V_1,\ldots,V_k)$ be a balanced partition of $V(G)$. Then 
\begin{align*}
    f_k(K_n) = e(V_1,\ldots,V_k) 
    &=\binom{n}{2} - s\binom{r + 1}{2} - (k - s)\binom{r}{2} \\
    &= \frac{k-1}{k}m + \frac{k-1}{2k}n - \frac{s(k - s)}{6}  \\
    &= \frac{k-1}{k}m + \frac{k-1}{2k}h(m)+ \frac{k-1}{2k}-\frac{s(k-s)}{2k}. \qedhere
\end{align*}
\end{proof}

For any integer $k$, the following lemma gives a tight lower bound for the maximum $k$-partite subgraph. 
\begin{lemma} \label{k-cut}
Let $G$ be a graph with $m$ edges. Let $k\geq 2$ be an integer. Then
\begin{itemize}
    \item[(i)] $f_k(G)\geq \frac{k-1}{k}m+\frac{k-1}{2k}h(m)+\frac{k-1}{2k}-\max\limits_{s \in [k-1] \cup \{0\}}\{\frac{s(k-s)}{2k}\}$, where the lower bound can be achieved by a balanced partition. 
    \item[(ii)] If $G$ is not a complete graph of odd order, then $f_2(G)\geq \frac{m}{2}+\frac{1}{4}h(m)+\frac{1}{4}$, where the lower bound can be achieved by a balanced partition. 
\end{itemize}
\end{lemma}
\begin{proof} 
We define 
\[q(m):=\frac{k-1}{k}m+\frac{k-1}{2k}h(m)+\frac{k-1}{2k}-\max_{s \in [k-1] \cup \{0\}}\{\frac{s(k-s)}{2k}\} \quad \text{and} \quad 
p(m):=\frac{q(m)}{m}.\] 
It is easy to show that $q'(m)\geq 0$ and $p'(m) \leq 0$. Identifying any two nonadjacent vertices of a simple graph yields a complete multigraph (with no loop). There is a natural bijection between complete multigraphs and weighted complete graphs with positive integer weight. Let $H$ be the weighted complete graph corresponding to $G$. Let $n_1=\left| H \right|$ and $m_1=\binom{n_1}{2}\leq w(H)$. Note that every partition of $H$ corresponds to a partition of $G$ with the same number of crossing edges. 

First we prove (i). Consider a random balanced partition of $H$ into $k$ parts. The expected number of edges in a maximum $k$-partite subgraph is
\[f_k(K_{n_1})\cdot \frac{w(H)}{m_1}\ge q(m_1)\cdot \frac{w(H)}{m_1}\ge q(w(H)),\]
where the first inequality follows from Lemma \ref{Kn-cut} and the second inequality follows from the fact that $p(m_1)\ge p(w(H))$.
Thus there must be some $k$-partite subgraph of $H$ of size at least $q(w(H))= q(m)$, achieved by a balanced partition of $H$ into $k$ parts.

Now we prove (ii). The proof of (ii) proceeds by induction on $m$. The base case $m=1$ is straightforward. We assume that the statement holds for all graphs with fewer than $m$ edges.
If $G$ is a complete graph of even order, then $f_2(G)=m/2+h(m)/4+1/4$ by Lemma \ref{Kn-cut}. We may therefore assume that $G$ is not a complete graph. Furthermore, assume $G$ has no isolated vertex. Since $G$ is not complete, there exists at least one edge $e\in E(H)$ with weight $w(e)\geq 2$. Let $W$ be the induced subgraph of $H$ such that $V(W)=\{v\in V(H):w_{H}(v)\geq n_1\}$, where $w_{H}(v)$ denotes the total weight incident to $v$ in $H$. Let $w_{W}(e):=w_{H}(e)-1$, where $w_{W}(e)$ denotes the weight of edge $e$ in $W$. Let $n_2=|W|\geq 2$ and $m_2=w(W)=m-m_1\geq 1$.
First assume $W$ is a complete graph on $n_2$ vertices with all edge weight 1 and odd $n_2$. If $n_1$ is even, by Lemma \ref{Kn-cut}, then we have
\begin{align*}
 f_2(G)& \geq f_2(K_{n_1})+f_2(K_{n_2}) \geq \frac{m_1}{2}+\frac{h(m_1)}{4}+\frac{m_2}{2}+\frac{h(m_2)}{4}+\frac{1}{4} \\
 & =\frac{m}{2}+\frac{h(m_1)+h(m_2)+1}{4}\geq \frac{m}{2}+\frac{h(m)}{4}+\frac{1}{4},
\end{align*}
where the last inequality follows from the convexity of the function $h(x)=\sqrt{2x+1/4}-1/2$.
Otherwise $n_1$ is odd. Then 
\begin{align*}
 f_2(G)&\geq f_2(K_{n_1})+f_2(K_{n_2}) = \frac{n_1-1}{2}\cdot \frac{n_1+1}{2}+\frac{n_2-1}{2}\cdot \frac{n_2+1}{2} =\frac{n_1^2-1}{4}+\frac{n_2^2-1}{4} \\
 &\geq \frac{1}{2}\binom{n_1}{2}+\frac{1}{2}\binom{n_2}{2}+\frac{1}{4}\left({\sqrt{2\binom{n_1}{2}+2\binom{n_2}{2}+\frac{1}{4}}-\frac{1}{2}}\right)+\frac{1}{4}\\
 &=\frac{m}{2}+\frac{1}{4}h(m)+\frac{1}{4}.
\end{align*}
So we may assume $W$ is not a complete graph of odd order. 
Let $(W_1,W_2)$ be a balanced partition of $W$ such that $e(W_1,W_2) \geq  m_2/2 + h(m_2)/4 + 1/4$ by inductive hypothesis. Then $f_2(G) \geq f_2(K_{n_1}) + e(W_1,W_2) \geq m_1/2 + h(m_1)/4 + m_2/2 + h(m_2)/4 + 1/4 \geq m/2 + h(m)/4 + 1/4$, which completes the proof. \qedhere
\end{proof}

Theorem \ref{maxcut k=3} is a direct corollary of Lemmas \ref{Kn-cut} and \ref{k-cut} since $\frac{2}{3}m+\frac{1}{3}h(m)+\frac{1}{3}-\max\limits_{s \in [2] \cup \{0\}}\{\frac{s(3-s)}{6}\}=\frac{2}{3}m+\frac{1}{3}h(m)$. Equality in \eqref{m3c} holds if and only if all inequalities in the proof of Lemma \ref{k-cut} are tight, which happens only when $G$ is a complete graph of order $3r+1$ or $3r+2$, by Lemma \ref{Kn-cut}. 

\section{Judicious partitions of graphs with few edges}
In this section, we answer Problem \ref{problem} affirmatively when $m<2k^2$. We begin with some definitions and lemmas. Let $k$ be an integer.

\begin{definition}
Let $(V_1,\ldots,V_k)$ be a $k$-partition of $V(G)$ with $e(V_1)\geq e(V_2)\geq \ldots \geq e(V_k)$. The \textit{lexicographic order} of \((V_1, \ldots, V_k)\), denoted by \(lex(V_1, \ldots, V_k)\), is the $k$-tuple $(e(V_1),\ldots,e(V_k))$. We say that \((V_1, \ldots, V_k)\) has \textit{smaller lexicographic order} than \((W_1, \ldots, W_k)\), denoted by $lex(V_1, \ldots, V_k)<lex(W_1, \ldots, W_k)$,
if and only if there exists a positive integer \(t \leq k\) such that
\begin{enumerate}
    \item[(i)] For every positive integer \(i < t\), $e(V_i)=e(W_i)$, and
    \item[(ii)] $e(V_t)<e(W_t)$.
\end{enumerate}
\end{definition}

\begin{definition}
Let $(V_1,\ldots,V_k)$ be a $k$-partition of $V(G)$ and let $t\geq 1$ be an integer. Let $H_i$ be a subgraph of $G$ and $X_i\notin V(H_i)$ be a vertex set for each $i\in[t]$. We say that $(V_1,\ldots,V_k)$ is 
\textit{$(X_1,H_1,X_2,H_2,\ldots,X_t,H_t)$-decreasing} if the following holds.
\begin{enumerate}
    \item[(i)] For each $j\in [t]$, moving $X_j$ to $H_{j}$ yields a $k$-partition with smaller lexicographic order.
    \item[(ii)] If, in addition, the number of the crossing edges does not decrease, then $(V_1,\ldots,V_k)$ is called \textit{$(X_1,H_1,X_2,H_2,\ldots,X_t,H_t)$-judicious}.
\end{enumerate}
\end{definition}

The next lemma is straightforward and is included here for completeness.
\begin{lemma} \label{l16}
Let $(V_1,\ldots,V_k)$ be a $k$-partition of $V(G)$ with minimum lexicographic order. Then for all $i, j \in [k]$, and for every $v\in V_i$, the following holds.
\begin{enumerate}
    \item[(i)] If $d_{G[V_i]}(v)\geq 1$, then $d_{G[V_j]}(v)\geq 1$.
    \item[(ii)] If $d_{G[V_i]}(v)\ge 2$ and $e(V_i)>e(V_j)$, then $d_{G[V_j]}(v)\geq 2$.
    \item[(iii)] If $d_{G[V_i]}(v)\ge 1$ and $e(V_i)>e(V_j)+1$, then $d_{G[V_j]}(v)\geq 2$.
\end{enumerate}
\end{lemma}

\begin{remark}
The condition of Lemma \ref{l16} can be replaced by: Let $(V_1,\ldots,V_k)$ be a $k$-partition of $V(G)$ maximizing $e(V_1,\ldots,V_k)$ with minimum lexicographic order.
\end{remark}

Now we proceed to the proof of Theorem \ref{small edges}. 
\begin{proof}[Proof of Theorem \ref{small edges}]
We divide the proof into three cases according to the size of $m$.
\medskip \begin{case}
$m<k^2/{2}+k/2$.

Note that $h(m)<k$ when $m<k^2/{2}+k/2$. By \eqref{min-eVi}, one can find a $k$-partition $(V_{1},\ldots,V_{k})$ satisfying (i). In particular, $\max_{1\le i \le k}\{e({{V}_{i}})\}\le m/k^2+(k-1)h(m)/(2k^2)<1/2+1/(2k)+(k-1)/(2k)=1$. Hence, $e({V}_{i}) =0$ for all $i\in [k]$ and so $e(V_1,\ldots,V_k)=m$. Thus, (ii) holds.
\end{case}
\medskip \begin{case}
$k^2/{2}+k/2 \le m < 3k^2/2$.

Let $(V_{1},\ldots,V_{k})$ be a partition maximizing $e(V_{1},\ldots,V_{k})$ with minimum lexicographic order. Then (ii) holds by \eqref{max-k-cut}. Define $x:=|\{i \in [k]: e(V_i)\ge 2\}|$, $y:=|\{i \in [k]: e(V_i)=1\}|$ and $z:=|\{i \in [k]: e(V_i)=0\}|$. So we have
\begin{equation} \label{l2e3}
    x+y+z=k.
\end{equation}
Note that $m/k^2+(k-1)h(m)/(2k^2)\ge 1$ when $m\ge k^2/{2}+k/2$. 
If $x=0$, then $\max_{1\le i \le k}\{e({{V}_{i}})\}\le 1 \le m/k^2+(k-1)h(m)/(2k^2)$ and we are done. Now suppose $x\ge1$. For brevity, denote by $X_i$, $Y_i$ and $Z_i$, the induced subgraph $G[V_i]$ for each $i$ satisfying $e(V_i)\ge 2$, $e(V_i)=1$ and $e(V_i)=0$, respectively. Note that each $X_i$ contains a copy of either $2K_2$ or $P_3$. 

\medskip \begin{claim} \label{c1}
$e(X_i,Y_j)\ge 5$, for every $i\in [x]$, $j \in [y]$. 

By Lemma \ref{l16}, we already have $e(X_i,Y_j)\ge 4$. Suppose, for contradiction, that $e(X_{i'},Y_{j'})=4$ for some $i' \in [x]$ and $j'\in[y]$. Let $v_1v_2,v_3v_4\in E(X_{i'})$ (possibly $v_2=v_3$) and $v_5v_6\in E(Y_{j'})$. Note that $e\left(\{v_1,v_2,v_3,v_4\},Y_{j'}\right)\geq 4$. Then $e\left(\{v_1,v_2,v_3,v_4\},Y_{j'}\right)=e(X_{i'},Y_{j'})=4$. Without loss of generality, assume $v_1v_6 \notin E(G)$. First suppose $ v_2 \neq v_3$. Then $(V_1,\ldots, V_k)$ is \( (\{v_2, v_4\}, Y_{j'}, v_6, X_{i'}) \)-judicious when $v_3v_6 \notin E(G)$ or \( (\{v_2, v_3\}, Y_{j'}, v_6, X_{i'}) \)-judicious when $v_3v_6 \in E(G)$. If \( v_2 = v_3 \), then $(V_1,\ldots, V_k)$ is \( (v_4, Y_{j'}, v_6, X_{i'})\)-judicious when $v_2v_6\notin E(G)$ or \( (v_2, Y_{j'}, v_6, X_{i'})\)-judicious when $v_2v_6\in E(G)$. In all cases we obtain a contradiction to the minimality of $lex(V_1,\ldots,V_k)$, and therefore \( e(X_i, Y_j) \geq 5 \). \hfill $\blacksquare$ 
\end{claim}

\medskip \begin{claim} \label{c2} The following holds.
\begin{enumerate}
\item[(i)] $e(X_i, Z_j) \geq 6$, for every $i \in [x]$, $j \in [z]$.
\item[(ii)]$e(X_i, X_j) \geq 4$, for distinct $i, j \in [x]$.
\item[(iii)]$e(Y_i, Y_j) \geq 4$, for distinct $i, j \in [y]$. 
\item[(iv)]$e(Z_i, Z_j) \geq 3$, for distinct  $i, j \in [z]$.
\item[(v)] $e(Y_i,Z_j)\ge 3$, for every $i\in [y]$, $j \in [z]$.
\end{enumerate}
\end{claim}
We postpone the proof to Appendix \ref{appendix A}. \hfill $\blacksquare$

By Claims \ref{c1}, \ref{c2}, and \eqref{l2e3}, we have 
\begin{align}\label{l2e4}
  m\ge & 5xy+6xz+3yz+4\binom{x}{2}+4\binom{y}{2}+3\binom{z}{2}+2x+y \notag \\
=&\frac{3}{2}(x^2+y^2+z^2+2xy+2yz+2xz)+\frac{1}{2}x^2+\frac{1}{2}y^2+(2x-1)y+(3x-\frac{3}{2})z \notag \\
\ge &\frac{3}{2}k^2+\frac{1}{2}+\frac{1}{2}y^2+y+\frac{3}{2}z 
%\ge &\frac{3}{2}k^2+\frac{1}{2} \notag \\
>\frac{3}{2}k^2,   \notag 
\end{align}
a contradiction.
\end{case}

\medskip \begin{case}
$3k^2/2\le m <2k^2$.

Let $(V_{1},\ldots,V_{k})$ be a partition with minimum lexicographic order. Define $w:=|\{i \in [k]: e(V_i)\ge 3\}|$, $x:=|\{i \in [k]: e(V_i)= 2\}|$, $y:=|\{i \in [k]: e(V_i)=1\}|$ and $z:=|\{i \in [k]: e(V_i)=0\}|$. So we have
\begin{equation} \label{l2e5}
    w+x+y+z=k.
\end{equation}
For convenience, for every $i$, write $W_i$, $X_i$, $Y_i$ and $Z_i$ for $G[V_i]$, if $e(V_i)\ge 3$, $e(V_i)=2$, $e(V_i)=1$ and $e(V_i)=0$, respectively. First we suppose $w\ge 1$.  
\medskip \begin{claim} \label{c7} The following holds.
\begin{itemize}
    \item[(i)] $e(W_i,W_j)\ge 4$, for distinct $i, j \in [w]$.
    \item[(ii)] $\min\{e(W_i,X_j),e(W_i,Y_k),e(W_i,Z_l)\}\ge 6$, for every $i\in[w]$, $j\in[x]$, $k\in[y]$, $l\in[z]$.
    \item[(iii)] $e(Z_i,Z_j)\ge 4$, for distinct $i, j \in [z]$.
    \item[(iv)] $e(Y_i,Z_j)\ge 4$, for every $i\in [y]$, $j \in [z]$.
\end{itemize}
\end{claim}
We postpone the proof to Appendix \ref{appendix A}. \hfill $\blacksquare$

By Claims \ref{c1}, \ref{c2}, \ref{c7} and (\ref{l2e5}), we have 
\begin{align}\label{l2e6}
 m\ge & 6wx+6wy+6wz+5xy+6zx+4yz+4\binom{w}{2}+4\binom{x}{2}+4\binom{y}{2}+4\binom{z}{2}+3w+2x+y \notag \\
=& 2(w+x+y+z)^2+w+(2w-1)y+2(w-1)z+xy+2xz+2wx \notag \\
\ge & 2k^2+1 > 2k^2,   \notag 
\end{align}
a contradiction. Thus, $w=0$. 
Note that $h(m)>\sqrt{3}k-1/2$ as $m\ge 3k^2/2$. Hence, $m/k^2+(k-1)h(m)/(2k^2)>3/2+(k-1)(\sqrt{3}k-1/2)/(2k^2)>2\ge \max_{1\le i \le k}\{e({{V}_{i}})\}$. Thus, Theorem \ref{small edges} holds. Now we have $w=0$. Suppose $x\ge 1$.

\medskip\begin{claim} \label{c41}
$e(X_i,X_j)\ge 6$, for distinct $i, j \in [x]$.
\end{claim}
We postpone the proof to Appendix \ref{appendix A}. \hfill $\blacksquare$

By Claims \ref{c2} and \ref{c41},
\begin{align}
     m\ge & 5xy+6xz+3yz+6\binom{x}{2}+4\binom{y}{2}+3\binom{z}{2}+2x+y \notag \\
     = & (\sqrt{3}x+\sqrt{2}y+\frac{3}{2\sqrt{2}}z)^2-(x+y+z)+\frac{3}{8}z^2+(5-2\sqrt{6})xy+\left((6-\frac{3\sqrt{6}}{2})x-\frac{1}{2}\right)z \notag \\
     \ge &(\sqrt{3}x+\sqrt{2}y+\frac{3}{2\sqrt{2}}z)^2-k. \notag
\end{align}
Thus, $\sqrt{m+k}\ge \sqrt{3}x+\sqrt{2}y+3z/(2\sqrt{2})=\sqrt{3}x+\sqrt{2}y+3(k-x-y)/(2\sqrt{2})\ge 3k/(2\sqrt{2})+(\sqrt{3}/{2}-{3}/(4\sqrt{2}))(2x+y),$ which implies
\begin{equation} \label{l2e7}
    2x+y\le \left(\sqrt{m+k}-\frac{3}{2\sqrt{2}}k\right)/\left(\frac{\sqrt{3}}{2}-\frac{3}{4\sqrt{2}}\right).
\end{equation}
Let $u_1(m):=m-\left(\sqrt{m+k}-{3k}/(2\sqrt{2})\right)/\left({\sqrt{3}}/{2}-{3}/(4\sqrt{2})\right)-\left((k-1)\sqrt{2m+1/4}\right)/{2}$. Since $u_1'(m)> 0$, we have $u_1(m)\ge u_1(3k^2/2)>-(k-1)/{4}-(k-2)^2/{8}$. By \eqref{l2e7}, we have $m-(2x+y)k-\left((k-1)\sqrt{2m+1/4}\right)/{2} \ge u_1(m) >-(k-1)/{4}-(k-2)^2/{8}$.
Thus, 
\begin{align}
e(V_1,\ldots,V_k)=& m-(2x+y) 
    = \frac{k-1}{k}m+\frac{1}{k}m-(2x+y)-\frac{k-1}{2k}\sqrt{2m+\frac{1}{4}}+\frac{k-1}{2k}\sqrt{2m+\frac{1}{4}}  \notag \\
    =&\frac{k-1}{k}m+\frac{1}{k}\left(m-(2x+y)k-\frac{k-1}{2}\sqrt{2m+\frac{1}{4}}\right)+\frac{k-1}{2k}\sqrt{2m+\frac{1}{4}} \notag \\
    \ge &\frac{k-1}{k}m+\frac{1}{k}\left(-\frac{k-1}{4}-\frac{(k-2)^2}{8}\right)+\frac{k-1}{2k}\sqrt{2m+\frac{1}{4}} \notag \\
    =&\frac{k-1}{k}m+\frac{k-1}{2k}h(m)-\frac{(k-2)^2}{8k}. \notag
\end{align} 
Hence, $(V_1,\ldots,V_k)$ is a partition satisfying Theorem \ref{small edges}. Now we have $x=0$.
\medskip \begin{claim} \label{c42} 
The following holds.
\begin{itemize}
    \item[(i)] $e(Y_i,Y_j)\ge 4$, for distinct $i, j \in [y]$. 
    \item[(ii)] $e(Y_i,Z_j)\ge 2$, for every $i\in [y]$, $j \in [z]$.
    \item[(iii)] $e(Z_i,Z_j)\ge 1$, for distinct $i, j \in [z]$.
\end{itemize}

The proof of (i) is the same as Claim \ref{c2} (iii). The proofs of (ii) and (iii) are direct conclusions of Lemma \ref{l16} and readers can verify them easily. \hfill $\blacksquare$
\end{claim}
By Claim \ref{c42}, we have
\begin{equation} \label{l2e8}
    m\ge 4\binom{y}{2}+2yz+\binom{z}{2}+y=2(y+\frac{1}{2}z)^2-(y+\frac{1}{2}z)
\end{equation}
and
\begin{equation} \label{l2e9}
    y+z=k.
\end{equation}
By \eqref{l2e8} and \eqref{l2e9}, we have
\begin{equation}
 y\le \frac{1}{2}+\frac{1}{2}\sqrt{8m+1}-k.   
\end{equation}
Let $u_2(m):=m-k\left(1/2+\sqrt{8m+1}/2-k\right)-\left((k-1)\sqrt{2m+{1}/{4}}\right)/2$. Since $u_2'(m)\geq 0$, we have $u_2(m)\ge u_2(3k^2/2)>-(k-1)/{4}-(k-2)^2/{8}$.
Thus, $m-ky-\left((k-1)\sqrt{2m+{1}/{4}}\right)/2\ge u_2(3k^2/2)>-(k-1)/{4}-(k-2)^2/{8}$. 
Hence, 
\begin{align}
    e(V_1,\ldots,V_k)=&m-y 
=\frac{k-1}{k}m+\frac{1}{k}(m-ky-\frac{k-1}{2}\sqrt{2m+\frac{1}{4}})+\frac{k-1}{2k}\sqrt{2m+\frac{1}{4}} \notag \\
>&\frac{k-1}{k}m+\frac{1}{k}(-\frac{k-1}{4}-\frac{(k-2)^2}{8})+\frac{k-1}{2k}\sqrt{2m+\frac{1}{4}}\notag \\
=&\frac{k-1}{k}m+\frac{k-1}{2k}h(m)-\frac{(k-2)^2}{8k}. \notag
\end{align}
Thus, $(V_1,\ldots,V_k)$ is a partition satisfying Theorem \ref{small edges}. This completes the proof. \qedhere
\end{case}
\end{proof}

\section{Some lemmas}
In this section we show some useful lemmas. These will be used in the next two sections to prove our main results. First, we need a lemma due to Xu and Yu~\cite{XY2009}.

\begin{lemma}[Xu and Yu~\cite{XY2009}] \label{l1}
Let ${m}'={(k-1)}^{2}m/{{k}^{2}}+q$. If $q\le (k-1)h(m)/(2k^2)$, then
$m'/(k-1)^2+(k-2)h(m')/(2(k-1)^2)\leq m/k^2+(k-1)h(m)/(2k^2)$.
\end{lemma}

 For a partition $(V_{1},\ldots,V_{k})$ of $V (G)$, we say that $V_{i}$ satisfies \textit{property $Q$} if $\left| N(x)\cap \overline{{{V}_{i}}} \right|\geq (k-1)|N(x)\cap V_i|$ for every $x\in V_i$. Note that moving any vertex of $V_i$ out does not destroy property $Q$. Moreover, it is easy to show that $e(V_i,\overline{V_i})\geq (k-1)e(V_i)$ if $V_i$ satisfies property $Q$.
 
\begin{definition} \label{l3}
Let $(V_1,\ldots,V_k)$ be a $k$-partition of $V(G)$. Let $m:=e(G)$. We call $(V_1,\ldots,V_k)$ is \textit{good} if all of the following holds.
\begin{enumerate}
    \item [(i)] $e({{V}_{1}},\ldots,{{V}_{k}})\ge f_k(m)$.
    \item [(ii)] $e(V_1)\geq \ldots \geq e(V_k)$.
    \item [(iii)] $V_{1}$ satisfies property $Q$.
\end{enumerate}
\end{definition}

Indeed, such a partition exists. Let $({{V}_{1}},\ldots,{{V}_{k}})$ be a $k$-partition maximizing $e(V_1,\ldots,{{V}_{k}})$ and so $e(V_1,\ldots,V_k)\ge f_k(m)$. In addition, we may assume that $e(V_1)\geq \ldots \geq e(V_k)$. For every $i,j\in [k]$ and all $x \in V_{i}$, we have $|N(x)\cap V_{j}|\ge |N(x)\cap V_{i}|$, otherwise we could move $x$ from $V_{i}$ to $V_{j}$ to obtain a larger number of crossing edges. Thus $\left| N(x)\cap \overline{{{V}_{i}}} \right|\geq (k-1)|N(x)\cap V_i|$ for all $x \in V_{i}$. 

The following lemma describes the edge distributions in the new partition obtained by moving a vertex from a good partition.

\begin{lemma} \label{l4}
Let $m$, $d$ and $k\geq 3$ be integers and let $\beta_1\in \mathbb{R}$. Let $(V_{1},\ldots,V_{k})$ be a good partition such that $lex(V_1,\ldots,V_k)$ is minimum. Let $e(V_{1})=m/k^2+\beta_{1}$ with $\beta_{1}>(k-1)h(m)/(2k^2)$. Let $v\in V_1$ with degree $d:=d_{G[V_1]}(v)>0$ in $G[V_1]$
such that $2e(V_1)\geq d(d+1)$. Let $Y_{1}=V_{1}\setminus \{v\}$, $x=e({{Y}_{1}},\overline{{{Y}_{1}}})$ and $m'=e(\overline{{Y_1}})$. Suppose that $G[\overline{Y_{1}}]$ has a $(k-1)$-partition $(Y_{2},\ldots,Y_{k})$ such that
\begin{equation} \label{l4e1}
    \max\limits_{2\le i\le k} e(Y_i)  \le \frac{m'}{(k-1)^2}+\frac{k-2}{2(k-1)^2}h({m}').
\end{equation}
Then the following holds.
\begin{enumerate}
    \item[(i)] $Y_1$ satisfies property $Q$.
    \item[(ii)] $x=\sum\limits_{u\in {{Y}_{1}}}{\left| N(u)\cap \overline{{{V}_{1}}} \right|}+\left| N(v)\cap {{V}_{1}} \right|
\ge  (k-1)\sum\limits_{u\in {{Y}_{1}}}{\left| N(u)\cap {{V}_{1}}\right|}+d$ and \\
    $2(k-1)(\frac{m}{{{k}^{2}}}+\beta_{1})-(k-2)d\le x \le \frac{{{k}^{2}}-1}{{{k}^{2}}}m-\beta_{1} +d$.
    \item[(iii)] ${m}'\le \frac{{{(k-1)}^{2}}}{{{k}^{2}}}m-(2k-1)\beta_{1} +(k-1)d$.
    \item[(iv)] $\underset{2\le i\le k}{\mathop{\max }}\,\{e(Y_{i})\}  \le \frac{m}{{{k}^{2}}}+\frac{k-1}{2{{k}^{2}}}h(m)$.
\end{enumerate}
\end{lemma}
\begin{proof}
First we show (i). Since
$\left| N(u)\cap \overline{Y_{1}} \right|\ge \left| N(u)\cap \overline{V_{1}} \right| \ge (k-1)\left| N(u)\cap {{V}_{1}} \right|\geq (k-1)\left| N(u)\cap {{Y}_{1}} \right|$ for every $u\in Y_{1}$, we have $Y_1$ satisfies property $Q$. Now we prove (ii).
Note that $(V_1,\ldots,V_k)$ is good, we have 
\begin{align}
 x&=\sum_{u\in {{Y}_{1}}}{\left| N(u)\cap \overline{{{Y}_{1}}} \right|} 
 =\sum_{u\in {{Y}_{1}}}{\left| N(u)\cap \overline{{{V}_{1}}} \right|}+\left| N(v)\cap {{V}_{1}} \right| \notag \\
 & \ge   (k-1)\sum_{u\in {{Y}_{1}}}{\left| N(u)\cap {{V}_{1}}\right|}+d
=2(k-1)e({{V}_{1}})-(k-2)d \notag \\
&=2(k-1)(\frac{m}{k^2}+\beta_{1})-(k-2)d. \notag  
\end{align}
On the other hand, 
\begin{align}
 x\le m-e({{Y}_{1}})=m-e(V_{1})+d=\frac{k^2-1}{k^2}m-\beta_{1} +d. \notag  
\end{align}
 So we have (ii). 
By (ii), we have ${m}'=m-e({{Y}_{1}})-x 
\le m-(m/k^2+\beta_{1}-d)-2(k-1)(m/k^2+\beta_{1})+(k-2)d= (k-1)^2m/k^2-(2k-1)\beta_{1} +(k-1)d$. Thus, (iii) holds. 

To prove (iv), by Lemma \ref{l1}, we only need to prove
\[
m'\le \frac{(k-1)^2}{k^2}m+\frac{k-1}{2k^2}h(m).
\]
Suppose $d \le h(m)/(2k^2)+(2k-1)\beta_{1}/(k-1) $. By (iii), ${m}'\le (k-1)^2m/k^2+(k-1)h(m)/(2k^2)$. Now suppose that $d >h(m)/(2k^2)+(2k-1)\beta_{1}/(k-1)$.
So $d>h(m)/k$ as $\beta_{1}>(k-1)h(m)/(2k^2)$. Let $r_1(y)=-(2k-1)y^2/{2}-y/2$. Since $r_1'(y)<0$ when $y\geq 0$, we have $r_1(y)<r_1(h(m)/k)=-(2k-1)h^2(m)/(2k^2)-{h(m)}/(2k)$ when $y>h(m)/k$.
Recall that $e({{V}_{1}})\ge ({d^{2}}+d)/2$ and thus $x\ge d +(k-1){d^{2}}$ by (ii). Then 
\begin{align}
m'&=m-(e(V_1)-d)-x \le m-\frac{{d^{2}}+d}{2} -(k-1){d^{2}} = m-\frac{(2k-1)}{2}d^{2}-\frac{d}{2} \notag \\
&<m-\frac{2k-1}{2k^2}h^2(m)-\frac{h(m)}{2k}=\frac{(k-1)^2}{k^2}m+\frac{k-1}{2k^2}h(m).\notag
\end{align}
The last equality holds since $2m=h^2(m)+h(m)$. This proves (iv).
\end{proof}

We also need the following technical lemma.

\begin{lemma} \label{l5}
Suppose \( m \geq 18 \) and \( \beta_1 > h(m)/9 \). Let \( d \geq 1 \) be an integer and \( c \geq 0 \) be a real number. Let the function \( \psi : \mathbb{R}^2 \to \mathbb{C} \) be
\[
\psi(x, c) := \frac{1}{4} \sqrt{\frac{16}{9}m - 2\beta_1 + 2d - 2x+ \frac{1}{4}}  + \frac{1}{2}x - \frac{2}{9}m - \frac{1}{2}(\beta_1 - d) - \frac{1}{3}\sqrt{2m + \frac{1}{4}} + \frac{1}{24} + c.
\]
Then the following holds.
\begin{enumerate}
    \item[(i)] \(\psi(\cdot, c)\) is concave for any fixed \( c \geq 0 \).
    
    \item[(ii)] \(\psi\left(8m/9 - \beta_1 + d, c\right) \geq 0\).
    
    \item[(iii)] If \(\beta_1 \geq h(m)/9 + 1/2\), then \(\psi\left(4m/9 + 4\beta_1 - d, c\right) \geq 0\).
    
    \item[(iv)] \(\psi\left(4m/9 + 4\beta_1 - d + 1, c\right) \geq 0\).
    
    \item[(v)] For \( c \geq 1/4 \), \(\psi\left(4m/9 + 4\beta_1 - d, c\right) \geq 0\).
\end{enumerate}
\end{lemma}
We defer the proof to Appendix \ref{appendix B}.

\section{Stability of judicious 2-partitions}
In this section, we study the stability of judicious $2$-partitions and present proofs of Theorem \ref{stability} and Theorem \ref{sta1}. We first prove Theorem \ref{sta1}, which covers the case of Theorem \ref{stability} when $G$ is not a complete graph of odd order. In fact, Theorem \ref{sta1} establishes a stronger judicious partition result if the maximum cut of a graph is larger than that of complete graphs. Here we sketch the proof of Theorem \ref{sta1}. We start with a good partition $(V_1,V_2)$ with minimum $lex(V_1,V_2)$. We show that if $e(V_1)$ is small then we are done. Otherwise, by moving a vertex with small degree from $V_1$ to $V_2$, we obtain a new partition $(W_1,W_2)$. We prove that $(W_1,W_2)$ satisfies the condition of the theorem.

\begin{theorem} \label{sta1}
If $f_2(G)\ge m/2+h(m)/4+1/4$, then $G$ admits a $2$-partition \(  (V_1 , V_2) \) such that 
\begin{enumerate}
    \item [\textup{(i)}]\label{t1} $\max\limits_{1\leq i\leq 2}e(V_i) \leq  \frac{m}{4}+\frac{h(m)}{8}$, and
    \item [\textup{(ii)}]\label{t2} $e(V_1, V_2) \geq \frac{m}{2}+\frac{h(m)}{4}+\frac{1}{4}.$
\end{enumerate}
\end{theorem}

\begin{proof}
Assume $G$ has no isolated vertex. Let \( V(G) = V_1 \cup V_2 \) be a good $2$-partition with minimum $lex(V_1,V_2)$. If \( e(V_1) \) satisfies (i), then we are done. Otherwise, suppose
\begin{equation} \label{t4}
e(V_1) = \frac{m}{4} + \alpha \quad \text{with} \quad \alpha > \frac{1}{8}h(m).   
\end{equation}
Since $V_1$ satisfies property $Q$, we have
$e(V_1, V_2) \geq m/2 + 2\alpha$,
and hence
$e(V_2) = m - e(V_1) - e(V_1, V_2) \leq m/4 - 3\alpha.$

Let $v_{0},v_{1}\in V_{1}$ such that $d_{G[V_{1}]}(v_{0})=d_{0}$ and $d_{G[V_{1}]}(v_{1})=d_{1}$, where $d_{0}$ and $d_{1}$ are the minimum and the second-minimum nonzero degrees in $G[V_{1}]$, respectively. Define a new $2$-partition $(W_1,W_2)$ such that \( W_1 = V_1 \setminus \{v_0\} \), \( W_2 = V_2 \cup \{v_0\} \). Then \( W_1 \) satisfies property $Q$ and \( e(W_1) < e(V_1) \). If $e(W_2)\leq e(W_1)$ and \((W_1, W_2)\) satisfies (ii), then \((W_1, W_2)\) is a good 2-partition with smaller $lex(W_1,W_2)$, a contradiction to our choice of $(V_1,V_2)$. So either \( e(W_2) > e(W_1) \) or \( (W_1, W_2) \) violates (ii).
Since \( e(V_{1}) \geq \binom{d_0+1}{2} \), it follows that
\begin{equation} \label{t7}
\frac{m}{4} + \alpha \geq \binom{d_0+1}{2}, \quad \text{which implies} \quad  d_0 \leq \sqrt{\frac{m}{2} + 2\alpha + \frac{1}{4}} - \frac{1}{2}.
\end{equation}
Assume $\alpha \geq \frac{h(m)+1}{8}$. Then 
\begin{align}
e(W_1, W_2) &= \sum_{x \in W_1} |N(x) \cap W_2| = \sum_{x \in W_1} |N(x) \cap V_2| + d_0  \geq \sum_{x \in W_1} |N(x) \cap V_1|  + d_0\notag \\ &= 2e(V_1) - d_0+ d_0 
= \frac{m}{2} + 2\alpha>  \frac{m}{2} + \frac{1}{4}h(m) + \frac{1}{4}. \notag
\end{align}
Thus, by \eqref{t7}, we have
\begin{align}
e(W_2) &= m - e(W_1) - e(W_1, W_2) \notag \\
  & \leq m - \left(\frac{m}{4} + \alpha - d_0\right) - \left(\frac{m}{2} + 2\alpha\right) \notag \\
  &= \frac{m}{4} - 3\alpha + d_0  \notag \\
  &\leq \frac{m}{4}+\sqrt{\frac{m}{2} + 2\alpha + \frac{1}{4}} - \frac{1}{2} - 3\alpha. \notag 
\end{align}
If $(W_1,W_2)$ does not satisfy (i), then
\begin{equation} \label{t8}
\min\{\alpha,\sqrt{\frac{m}{2} + 2\alpha + \frac{1}{4}} - \frac{1}{2} - 3\alpha\}> \frac{h(m)}{8}.     
\end{equation}
The left hand side is maximized when $\alpha=h(m)/8$, which contradicts \eqref{t8}. Thus $(W_1,W_2)$ satisfies (i). Thus, we may assume
\begin{equation} \label{t9}
    \frac{h(m)}{8} < \alpha < \frac{h(m)+1}{8}.
\end{equation}
Suppose there exists a vertex $v\in W_{1}$ such that $|N(v) \cap V_2| \geq |N(v) \cap V_1|+1$. Then
\begin{align}
e(W_1, W_2) &  \geq \sum_{x \in W_1} |N(x) \cap V_1| + d_0+1= \frac{m}{2} + 2\alpha + 1 >  \frac{m}{2} + \frac{1}{4}h(m) + \frac{1}{4}. \notag
\end{align}
By the same argument, we have $(W_1, W_2)$ satisfies (i).

Hence, we may assume
\[
|N(x) \cap V_2| = |N(x) \cap V_1| \quad \text{for all } x \in W_1.
\]
If $d_1<4\alpha$, then by \eqref{t9}, \( e(G[V_2 \cup \{v_1\}])< m/4-3\alpha + d_1 < m/4+\alpha=e(V_1)<m/4+h(m)/8+1/8 \). Thus $e(G[V_2 \cup \{v_1\}])<m/4+h(m)/8$ since $e(G[V_2 \cup \{v_1\}])\in \mathbb{N}$. Note that $e(G[V_1\setminus\{v_1\}])<m/4+h(m)/8+1/8-1<m/4+h(m)/8$ and $e(V_1 \setminus \{v_1\}, V_2 \cup \{v_1\})=e(V_1,V_2)\geq m/2+h(m)/4+1/4$.
Thus, \( (V_1 \setminus \{v_1\}, V_2 \cup \{v_1\}) \) is a desired partition.

Now suppose
\[
 d_1 \geq 4\alpha > \frac{h(m)}{2}.
\]
As \( 2e(V_1) > (|V_1| - 1)d_1 \), we have
\[
|V_1| - 1 - d_1 < \frac{2e(V_1)}{d_1} - d_1 < \frac{\frac{m}{2} + \frac{1}{4}h(m) + \frac{1}{4}}{h(m)/2} - \frac{h(m)}{2} = \frac{1}{2h(m)} + 1 < 2,
\]
which implies \( d_1 = |V_1| - 1 \) or \( d_1 = |V_1| - 2 \) since $d_1 \in \mathbb{N}$. If \( d_1 = |V_1| - 1 \), then \( G[V_1] \) is complete, and
\[
d_0 = d_1 \geq 4\alpha > \sqrt{\frac{m}{2} + 2\alpha + \frac{1}{4}} - \frac{1}{2},
\]
a contradiction to \eqref{t7}.
Thus,
 \( d_1 = |V_1| - 2 \).   
We claim that $v_0$ is adjacent to $v_1$. Otherwise, we have 
$m/2 + 2\alpha = 2e(V_1) \geq d_1(d_1 + 1)$, which implies
$ d_1 \leq \sqrt{m/2 + 2\alpha + 1/4} - 1/2< 4\alpha$.
Thus, $v_0$ is adjacent to every vertex of degree $|V_1|-2$ in $G[V_1]$. Thus $d_0=|V_1|-1$, again a contradiction to \eqref{t7}. This completes the proof of the theorem.
\end{proof}

\begin{proof}[Proof of Theorem \ref{stability}]
Assume $G \neq K_n$, where $n$ is odd.
By Lemma \ref{k-cut} (ii), we have $f_2(G) \ge m/2+h(m)/4+1/4$. This completes our proof by Theorem \ref{sta1}.
\end{proof}

\section{Main result: Tight bounds for judicious 3-partitions}
In this section, we prove tight bounds for judicious 3-partitions, i.e. Theorem \ref{main result}. When $m<18$, it follows immediately by Theorem \ref{small edges}. Now suppose $m \geq 18$. The proof is similar to that of Theorem \ref{sta1}. We start with a good $3$-partition $(V_1,V_2,V_3)$. We then modify the partition by moving a vertex from $V_1$ to form $W_1$, and apply Theorem~\ref{stability} to $\overline{W_1}$ to obtain a 2-partition $(W_2,W_3)$. We show that if $(V_1,V_2,V_3)$ does not satisfy Theorem \ref{main result}, then $(W_1,W_2,W_3)$ is a desired partition. 
\begin{proof}[Proof of Theorem \ref{main result}]
Assume that $G$ has no isolated vertex. Let $({{V}_{1}},V_2,{{V}_{3}})$ be a good $3$-partition of $V(G)$ with minimum $lex({{V}_{1}},V_2,{{V}_{3}})$. Let $e(V_{i})={m}/{k^2}+\beta_{i}$ for $i \in [3]$. Thus,
			\begin{equation} \label{t3e3}
				\beta_{1}>\frac{1}{9}h(m),
			\end{equation}
for otherwise, $({{V}_{1}},V_2,{{V}_{3}})$ is a desired partition since $e(V_1,V_2,V_3)\geq 2m/3+h(m)/3$. Let $v_{0}\in V_{1}$ such that $d_0:=d_{G[V_{1}]}(v_{0})=\min\{d_{G[V_1]}(v):d_{G[V_1]}(v)>0,v\in V_1\}$. Let $W_{1}=V_{1}\setminus v_{0}$, $x=e({{W}_{1}},\overline{W_1})$ and $m'=e(\overline{{{W}_{1}}})$. Applying Theorem \ref{stability} on $\overline{W_1}$, we have $\overline{{{W}_{1}}}$ admits a $2$-partition $({{W}_{2}},{{W}_{3}})$ satisfying both
			\begin{equation} \label{t3e4}
				e({{W}_{2}},W_3)\ge \frac{1}{2}{m}'+\frac{1}{4}h({m}')+c(\overline{W_1})
			\end{equation}
			and 
			\begin{equation} \label{t3e5}
				\mathop{\max}\limits_{2\le i \le 3}\{e({{W}_{i}})\}\le \frac{{{m}'}}{4}+\frac{1}{8}h({m}'),
			\end{equation}
where $c(\overline{W_1}) = 1/4$ if $\overline{W_1}$ is not a complete graph of odd order and 0 otherwise.
\medskip \begin{claim} \label{c7.2}
If $({{W}_{1}},W_2,{{W}_{3}})$ satisfies (ii), then $({{W}_{1}},W_2,{{W}_{3}})$ satisfies (i).
    
For otherwise, suppose that $\max_{1\leq i\leq 3} \{e({{W}_{i}})\}>{m}/{9}+h(m)/9$. By Lemma \ref{l4} (iv), we have $e({{W}_{1}})>{m}/{9}+h(m)/9\ge \max_{2\leq i\leq 3}\{e({{W}_{i}})\}$. Note that $W_1$ satisfies property $Q$ and $e(W_1)<e(V_1)$.
 Hence, $(W_{1},W_2,W_{3})$ is good with $lex(W_1,W_2,W_3)<lex(V_1,V_2,V_3)$, a contradiction. Thus, $\max_{1\leq i\leq 3}\{e({{W}_{i}})\}\le {m}/{9}+h(m)/9$.    \hfill $\blacksquare$
\end{claim}

\medskip Now
$e({{W}_{1}},W_2,{{W}_{3}})=e({{W}_{1}},{{{\overline{W}}}_{1}})
+e({{W}_{2}},{{W}_{3}}) 
				\ge x+m'/2+h({m}')/4+c(\overline{W_1})$ by \eqref{t3e4}.
We will show that $x+m'/2+h({m}')/4+c(\overline{W_1}) \geq 2m/3+h(m)/3$. Let 
			\begin{equation} \label{t3e6}
				g(x)=x+\frac{1}{2}m'+\frac{1}{4}h({m}')+c(\overline{W_1})-\frac{2}{3}m-\frac{1}{3}h(m).
			\end{equation}
Since $m'=m-e(W_{1})-x$ and $e(W_{1})=m/9+\beta_{1}-d_{0}$, we obtain
			\begin{equation}   \label{f(x)}
				g(x) =\frac{1}{2}x-\frac{2}{9}m-\frac{1}{2}(\beta_{1}-d_{0})-\frac{1}{3}\sqrt{2m+\frac{1}{4}} +\frac{1}{4}\sqrt{2\left(\frac{8}{9}m-\beta_{1} +d_{0}-x\right)+\frac{1}{4}}+\frac{1}{24}+c(\overline{W_1}). 
			\end{equation}
            
To prove $(W_1,W_2,W_3)$ satisfies (ii), it is equivalent to prove $g(x)\geq 0$. 
By Lemma \ref{l4} (ii), we have $a\leq x\leq b$, where $a:=4m/9+4\beta_{1}-d_{0}$ and $b:=8m/9-\beta_{1} +d_{0}$. By Lemma \ref{l5} (i), $g(x)$ is concave. So it suffices to prove $g(a)\geq 0$ and $g(b) \geq 0$. Note that Lemma \ref{l5} (ii) yields that $g(b)\ge 0$. First, suppose $G[\overline{W_1}]$ is not a complete graph of odd order and thus $c(\overline{W_1})=1/4$ by Theorem \ref{stability}. By Lemma \ref{l5} (v), we have $g(a)\ge 0$. Hence, $(W_1,W_2,W_3)$ satisfies Theorem \ref{main result}.

So we may assume that $G[\overline{W_1}]$ is a complete graph of odd order. In this case, $\overline{V_1}$ is a complete graph of even order. Since $lex(V_1,V_2,V_3)$ is minimum, we have $|V_2|=|V_3|$ and thus $e(V_2)=e(V_3)$.
If $\sum_{u\in {{W}_{1}}}{\left| N(u)\cap \overline{{{V}_{1}}} \right|}
\ge 2\sum_{u\in {{W}_{1}}}{\left| N(u)\cap {{V}_{1}}\right|}+1$, then we have $x\geq a+1$ by Lemma \ref{l4} (ii). By Lemma \ref{l5} (iv), we have $g(a+1)\geq 0$. Thus, $(W_1,W_2,W_3)$ satisfies both (i) and (ii). Now we may assume
\begin{equation} \label{eqnei}
 \sum\limits_{u\in {{W}_{1}}}{\left| N(u)\cap \overline{{{V}_{1}}} \right|}
=  2\sum\limits_{u\in {{W}_{1}}}{\left| N(u)\cap {{V}_{1}}\right|}.   
\end{equation}
By Lemma \ref{l5} (iii), $g(a)\ge 0$ when $\beta_1 \ge h(m)/9+1/2$. Thus, we may assume
\begin{equation} \label{alpha}
\beta_1 <\frac{1}{9}h(m)+\frac{1}{2}.
\end{equation}

By (\ref{eqnei}) and property $Q$, 
we have 
\[
 2\left| N(u)\cap {{V}_{1}} \right| = \left| N(u)\cap \overline{V_{1}} \right|\quad \text{for every $u \in {W_1}$}.
\]
Let $v_1\in W_1$ of degree $d_1:=d_{G[V_1]}(v)$ in $G[V_1]$. By symmetry of $V_2$ and $V_3$, we may assume $|N(v_1)\cap V_2|\le d_1$. We claim that $(V_1\setminus \{v_1\}, V_2\cup \{v_1\},V_3)$ is a 3-partition satisfying both (i) and (ii). Indeed, $e(V_1\setminus \{v_1\}, V_2\cup \{v_1\},V_3)=e(V_1,V_2,V_3)-|N(v_1)\cap V_2|+d_1\ge e(V_1,V_2,V_3)\geq 2m/3+h(m)/3$. Note that $|N(v_1)\cap \overline{V_1}|\leq |V_2|+|V_3|=|N(v_0)\cap \overline{V_1}|$, we have $\max\{e(V_2\cup \{v_1\}),e(V_3)\}\leq \max_{2\le i \le 3}\{e({{W}_{i}})\}\le m'/{4}+h({m}')/8\leq{m}/{9}+{h(m)}/{9}$ by (\ref{t3e5}) Lemma \ref{l4} (iv). Moreover, $e(V_1\setminus \{v_1\})=e(V_1)-d_1={m}/{9}+\beta_1-d_1<{m}/{9}+h(m)/9+1/2-1<{m}/{9}+h(m)/9$. Hence, $(V_1\setminus \{v_1\}, V_2\cup \{v_1\},V_3)$ is a desired $3$-partition. This completes the proof.
\end{proof}

\section*{Acknowledgments}
We thank Jie Ma for helpful discussion.

%\bibliography{references} 

{\small 
\begin{thebibliography}{99}

\bibitem{BS1999}
B. Bollob{\'a}s, A. D. Scott, Exact bounds for judicious partitions of graphs, Combinatorica 19 (1999) 473--486.

\bibitem{BS2002}
B. Bollob{\'a}s, A. D. Scott, Better bounds for Max Cut, in: Contemporary Combinatorics, in: Bolyai Soc. Math. Stud., vol. 10,
János Bolyai Math. Soc., Budapest, 2002, pp. 185--246.

\bibitem{BS2002problems}
B. Bollob{\'a}s, A. D. Scott, Problems and results on judicious partitions, Random Structures and Algorithms 21 (2002) 414--430.

\bibitem{BS2004}
B. Bollob{\'a}s, A. D. Scott, Judicious partitions of bounded-degree graphs, J. Graph Theory 46 (2004) 131--143.

\bibitem{Edwards1973}
C. S. Edwards, Some extremal properties of bipartite subgraphs, Canad. J. Math. 25 (1973) 475--485.

\bibitem{Edwards1975}
C. S. Edwards, An improved lower bound for the number of edges in a largest bipartite subgraph, in: Proc. 2nd Czechoslovak Symposium on Graph Theory, Prague, 1975, pp. 167--181.

\bibitem{FHZ2014}
G. Fan, J. Hou, Q. Zeng, A bound for judicious $k$-partitions of graphs, Discret. Appl. Math. 179 (2014) 86--99.

\bibitem{FMM2002}
U. Feige, M. Karpinski, M. Langberg, Improved approximation of Max-Cut on graphs of bounded degree, J. Algorithms 43 (2002) 201--219.

\bibitem{GM1995}
M. X. Goemans, D. P. Williamson, Improved approximation algorithms for maximum cut and satisfiability problems using semidefinite programming, J. ACM 42 (1995) 1115--1145.

\bibitem{GNYZ2025}
G. Gutin, M. A. Nielsen, A. Yeo, Y. Zhou, Judicious Partitions in Edge-Weighted Graphs with Bounded Maximum Weighted Degree, Preprint, arXiv:2507.05827[math.CO], 2025.

\bibitem{K2009}
R. M. Karp, Reducibility among combinatorial problems, in: Complexity of Computer Computations, (R. Miller and J. Thatcher, eds), Plenum Press, New York, 1972, pp. 219--241.

\bibitem{LX2016}
M. Liu, B. Xu, On judicious partitions of graphs, J. Comb. Optim. 31 (2016) 1383--1398.

\bibitem{scott2005}
A. Scott, Judicious Partitions and Related Problems, in: Surveys in Combinatorics, in: London Math. Soc. Lecture Note Ser., vol. 327, Cambridge Univ. Press, Cambridge, 2005, pp. 95--117.

\bibitem{SS1994}
F. Shahrokhi, L. A. Sz{\'e}kely, The complexity of the bottleneck graph bipartition problem, J. Combin. Math. Combin. Comput. 15 (1994) 221--226.

\bibitem{XY2009}
B. Xu, X. Yu, Judicious $k$-partitions of graphs, J. Combin. Theory Ser. B 99 (2009) 324--337.

\bibitem{XY2011}
B. Xu, X. Yu, Better bounds for $k$-partitions of graphs, Combin. Probab. Comput. 20 (2011) 631--640.

\end{thebibliography}
}

\begin{appendices}
\section{Proof of Claims \ref{c2}, \ref{c7} and \ref{c41}} \label{appendix A}

The proofs of Claims \ref{c2}, \ref{c7} and \ref{c41} are similar to that of Claim \ref{c1}.
\begin{proof}[Proof of Claim \ref{c2}]
\begin{enumerate} 
    \item [(i)] 
It is easy to show $e(X_i, Z_j) \geq 6$ by Lemma \ref{l16} (iii).

    \item [(ii)]
By Lemma \ref{l16} (i), $e(X_i,X_j)\ge 3$ . Now we assume $e(X_{i'},X_{j'})=3$ for some $i' \ne j' \in [x]$. Let $v_1v_2,v_2v_3\in E(X_{i'})$ and $v_4v_5,v_5v_6\in E(X_{j'})$. Then $(V_1,\ldots,V_k)$ is $(v_2,X_{j'},v_5,X_{i'})$-judicious.
Hence, $e(X_i,X_j)\ge4$.

    \item [(iii)]
By Lemma \ref{l16} (i), we have $e(Y_i,Y_j)\ge 2$. Suppose there exists some $i' \ne j' \in[y]$ such that $e(Y_{i'},Y_{j'})\leq 3$. Let $u_1u_2\in E(Y_{i'})$ and $u_3u_4\in E(Y_{j'})$.
It is easy to show that $E(\{u_1,u_2\},\{u_3,u_4\})\\=E(Y_{i'},Y_{j'})$. Without loss of generality, let $u_1u_3\in E(G)$. First suppose $d_{Y_{j'}}(u_1)=1$. Then $(V_1,\ldots,V_k)$ is $(u_2,Y_{j'},u_4,Y_{i'})$-judicious. Thus $d_{Y_{j'}}(u_1)=2$ and so $d_{Y_{j'}}(u_2)=1$. 
Thus $u_1u_4\in E(G)$. Then $(V_1,\ldots,V_k)$ is $(u_2,Y_{j'},u_4,Y_{i'})$-judicious if $u_2u_4\in E(G)$ and $(u_2,Y_{j'},u_3,Y_{i'})$-judicious otherwise. Thus, $e(Y_i,Y_j)\ge 4$.

    \item [(iv)] 
Let $p_1\in V(X_1)$ with $d_{X_1}(p_1)>0$. Suppose $e(Z_{i'},Z_{j'})\le 2$ for some $i'\ne j' \in [z]$.
By Lemma \ref{l16} (iii), $d_{Z_{i'}}(p_1)\ge 2$. If $e(N_{Z_{i'}}(p_1),Z_{j'})=0$, then $(V_1,\ldots,V_k)$ is $(p_1,Z_{i'},N_{Z_{i'}}(p_1),Z_{j'})$-judicious. Thus, $e(N_{Z_{i'}}(p_1),Z_{j'})\ge 1$. Let $p_2\in N_{Z_{i'}}(p_1)$ and let $p_2p_4\in E(G)$ where $p_4 \in V(Z_{j'})$. Then $(V_1,\ldots,V_k)$ is $(p_1,Z_{i'},N_{Z_{i'}}(p_1)-p_2,Z_{j'})$-judicious. Thus, $e(Z_i,Z_j)\ge 3$.

    \item [(v)] 
By Lemma \ref{l16} (i), $e(Y_i,Z_j)\ge 2$.  Suppose for contradiction that \( e(Y_{i'}, Z_{j'}) = 2 \) for some \( i' \in [y] \), \( j' \in [z] \). Let \( w_2w_3\in E(Y_{i'}) \). If \( N_{Z_{j'}}(w_2) \cap N_{Z_{j'}}(w_3) = \emptyset \), then $(V_1,\ldots,V_k)$ is $(w_2,Z_{j'},N_{Z_{j'}}(w_2),Y_{i'})$-judicious. Hence, there exists \( w_4 \in V(Z_{j'}) \) such that $w_2w_4,w_3w_4\in E(G)$. Let \( w_1 \in V(X_1)\) with $d_{X_1}(w_1)>0$. Then $(V_1,\ldots,V_k)$ is $(w_1,Z_{j'},N_{Z_{j'}}(w_1)\setminus {w_4},Y_{i'})$-judicious. Thus, $e(Y_i,Z_j)\ge 3$. \qedhere
\end{enumerate}
\end{proof}

\begin{proof}[Proof of Claim \ref{c7}]
\begin{enumerate} 
    \item [(i)]
If $\max\{|\{v\in W_i:d_{W_i}(v)\geq 1\}|,|\{v\in W_j:d_{W_j}(v)\geq 1\}|\}\ge 4$ for every $i\ne j \in [w]$, then we are done by Lemma \ref{l16} (i). Suppose there exist some $i' \ne j' \in [w]$ such that $|\{v\in W_{i'}:d_{W_{i'}}(v)>0\}|=|\{v\in W_{j'}:d_{W_{j'}}(v)>0\}|=3$. Let $v_1,v_2,v_3\in V(W_{i'})$ and $v_4,v_5,v_6\in V(W_{j'})$ have nonzero degree in $W_{i'}$ and $W_{j'}$, respectively.
If $e(W_{i'},W_{j'})=3$, then $E(W_{i'},W_{j'})=e(\{v_1,v_2,v_3\},\{v_4,v_5,v_6\})$ and they form a matching. Assume $v_1v_4 \in E(G)$.
In this case, $(V_1,\ldots,V_k)$ is $(v_1,W_{j'},v_4,W_{i'})$-decreasing. Thus, $e(W_i,W_j)\ge 4$. 

    \item [(ii)]
$\min\{e(W_i,Y_k),e(W_i,Z_l)\}\ge 6$ is a straightforward conclusion of Lemma \ref{l16} (iii). Now we prove $e(W_i,X_j)\geq 6$. Let $W_i^*$ be the subgraph of $W_i$ induced on all nonzero-degree vertices of $V(W_i)$.
If $e(W_i)>e(X_j)+1$, then we are done by Lemma \ref{l16} (iii) since $|W_i^*|\geq 3$. Now consider $e(W_i)=e(X_j)+1=3$. By Lemma \ref{l16} (i) and (ii), we need to show $2|\{v\in W_i^*:d_{W_i^*}(v)\geq 2\}|+|\{v\in W_i^*:d_{W_i^*}(v)= 1\}|\geq 6$.
If $|W_i^*|=3$, then $W_i^*=K_3$ and we are done. If $|W_i^*|=5$, then $W_i^*$ is $P_2\sqcup P_3$. If $|W_i^*|=6$, then $W_i$ consists of three independent edges. Now we consider $|W_i^*|=4$. If $W_i^*=P_4$ then are done. Otherwise, $W_i^*=K_{1,3}$. Let $x$ be the center vertex of $K_{1,3}$. Note that $e(W_i,X_j)\geq 5$. Assume $e(W_i,X_j)=5$. Let $xa\in E(G)$ and $xb\in E(G)$, where $a,b \in V(X_j)$. Then $(V_1,\ldots,V_k)$ is $(x,X_j,\{a,b\},W_i)$-decreasing.

\item[(iii)] 
By Claim \ref{c2}, $e(Z_i,Z_j)\ge 3$. Now suppose there exist some $i'\ne j' \in[z]$ such that $e(Z_{i'},Z_{j'})=3$. Let $p_1 \in V(W_1)$ with $d_{W_1}(p_1)>0$. Let $p_2\in N_{Z_{i'}}(p_1)$ and $p_2p_4\in E(G)$, where $p_4\in V(Z_{j'})$. Then $(V_1,\ldots,V_k)$ is $(p_1,Z_{j'},N_{Z_{i'}}(p_1)\setminus p_4,Z_{j'})$-decreasing. Thus, $e(Z_i,Z_j)\ge 4$. 

\item[(iv)] 
By Claim \ref{c2}, $e(Y_i,Z_j)\ge 3$. Now suppose there exist some $i'\in [y], j' \in[z]$ such that $e(Y_{i'},Z_{j'})=3$. Let $w_1 \in V(W_1)$ with $d_{W_1}(p_1)>0$. Let $w_2w_3\in E(Y_{i'})$. It is easy to show that there exists a vertex $w_4\in V(Z_{j'})$ such that $w_4 \in N_{Z_{j'}}(w_2)\cap N_{Z_{j'}}(w_3)$. 
Thus, $(V_1,\ldots,V_k)$ is $(w_1,Z_{j'},N_{Z_{j'}}(w_1)\setminus w_4,Y_{i'})$-decreasing. Hence, $e(Y_i,Z_j)\ge 4$. \qedhere
\end{enumerate}
\end{proof}

\begin{proof}[Proof of Claim \ref{c41}]
Since \( \max_{1 \leq i \leq k} e(V_i) = 2 \), each \( X_i \) is either a copy of \( 2K_2 \) or \( P_3 \) (modulo isolated vertices). By Claim \ref{c2}, we already have $e(X_i,X_j)\geq 4$. Suppose, for contradiction, that \( e(X_{i'}, X_{j'}) = t \) for some \( i' \neq j' \in [x] \), where \( t \in \{4,5\} \). We consider three cases based on the types of \( X_{i'} \) and \( X_{j'} \) (modulo isolated vertices).

\begin{enumerate}
    \item[\textup{(1)}] \textit{Both are \( 2K_2 \).} 
    Let \( u_1u_2, u_3u_4 \in E(X_{i'}) \) and \( u_5u_6, u_7u_8 \in E(X_{j'}) \). Without loss of generality, let $u_1u_5\in E(G)$. If $t=4$, then $G[u_1,u_2,u_3,u_4,u_5,u_6,u_7,u_8]$ is a copy of $C_8$ or $2C_4$. Thus one can find a partition with a smaller lexicographic order that contradicts our choice of $(V_1,\ldots,V_k)$. Now suppose $t=5$. Suppose there exists $u_9u_{10}\in E(G)$ such that $u_9\notin \{u_1,u_2,u_3,u_4\}$. If $u_{10}\notin \{u_5,u_6,u_7,u_8\}$, then $X_{i'}\cup X_{j'}$ (modulo isolated vertices) contains a copy of $C_8$ or $2C_4$, a contradiction again. Thus, $u_{10}\in \{u_5,u_6,u_7,u_8\}$. First suppose $u_{10}=u_5$. Then $(V_1,\ldots,V_k)$ is $(\{u_1,u_4,u_9\},X_{j'},\{u_5,u_8\},X_{i'})$-decreasing. Assume $u_{10}=u_6$. Then $(V_1,\ldots,V_k)$ is $(\{u_2,u_4\},X_{j'},\{u_5,u_8\},X_{i'})$-decreasing when $u_3u_7 \in E(G)$ and $(\{u_1,u_3\},X_{j'},\{u_5,u_7\},X_{i'})$-decreasing when $u_3u_7 \notin E(G)$. If $u_{10}=u_7$, then $(V_1,\ldots,V_k)$ is $(\{u_2,u_4\},X_{j'},\{u_6,u_8\},X_{i'})$-decreasing. If $u_{10}=u_8$, then $(V_1,\ldots,V_k)$ is $(\{u_2,u_4\},X_{j'},\{u_6,u_7\},X_{i'})$-decreasing. Thus, $u_9\in \{u_1,u_2,u_3,u_4\}$, which implies  \( E(X_{i'}, X_{j'}) = E(\{u_1, u_2, u_3, u_4\}, \{u_5, u_6, u_7, u_8\}) \). Without loss of generality, assume $d_{X_{j'}}(u_1)=2$ and $u_1u_{l_1},u_1u_{l_2}\in E(G)$, where $u_{l_1},u_{l_2}\in \{u_5,u_6,u_7,u_8\}$. Then $(V_1,\ldots,V_k)$ is $(\{u_1,u_3\},X_{j'}, \{u_{l_1},u_{l_2}\},X_{i'})$-decreasing.
   
    \item[\textup{(2)}] \textit{One is \( 2K_2 \) and the other is \( P_3 \).} Without loss of generality, assume \( X_{i'} = 2K_2 \) with edges \( u_1u_2, u_3u_4 \), and \( X_{j'} = P_3 \) with edges \( u_5u_6, u_6u_7\) (modulo isolated vertices). If \(t = 4\), then \((V_1, \ldots, V_k)\) is \((\{u_2, u_4\}, X_{j'}, u_6, X_{i'})\)-decreasing if $u_3u_6\notin E(G)$, and \((\{u_2, u_4\}, X_{j'}, \{u_5, u_7\}, X_{i'})\)-decreasing if $u_3u_6\in E(G)$. Suppose \(t = 5\). We can prove $E(X_{i'}, X_{j'}) = E(\{u_1, u_2, u_3, u_4\}, \{u_5,\\ u_6, u_7\})$ like (i). Without loss of generality, assume $d_{X_{j'}}(u_1)=2$. If $u_1u_6\notin E(G)$, then $(V_1,\ldots,V_k)$ is $(\{u_1,u_3\},X_{j'},  \{u_{5},u_{7}\},X_{i'})$-decreasing. Now assume $u_1u_6\in E(G)$. First suppose $d_{X_{j'}}(u_6)=3$. Then $(V_1,\ldots,V_k)$ is $(\{u_2,u_4\},X_{j'},u_6,X_{i'})$-decreasing when $u_3u_6\notin E(G)$, $(\{u_2,u_3\},X_{j'},u_6,X_{i'})$-decreasing when $u_4u_6\notin E(G)$ and $(\{u_1,u_4\},X_{j'},u_6,X_{i'})$-decreasing when $u_2u_6\notin E(G)$. If $d_{X_{j'}}(u_6)\leq 2$, then $(V_1,\ldots,V_k)$ is $(\{u_2,u_4\},X_{j'},\{u_5,u_7\}, X_{i'})$-decreasing.
    
   \item[\textup{(3)}] \textit{Both are \( P_3 \).} Let \( u_1u_2, u_2u_3 \in E(X_{i'}) \) and \( u_4u_5, u_5u_6 \in E(X_{j'}) \). Like (i), we can show $E(X_{i'}, X_{j'}) = E(\{u_1, u_2, u_3\}, \{u_4, u_5, u_6\})$. If $d_{X_{j'}}(u_1)=3$, then \( (V_1, \ldots, V_k) \) is \( (u_2, X_{j'}, u_5, X_{i'}) \)-decreasing. Now consider $d_{X_{j'}}(u_1)=2$. If $u_1u_5\notin E(G)$, then \( (V_1, \ldots, V_k) \) is \( (u_2, X_{j'}, u_5, X_{i'}) \)-decreasing. Otherwise, \( (V_1, \ldots, V_k) \) is \( (u_2, X_{j'}, \{u_4, u_6\}, X_{i'}) \)-decreasing if $d_{X_{j'}}(u_2)=2$, and \( (u_2, X_{j'}, u_5, X_{i'}) \)-decreasing if $d_{X_{j'}}(u_2)=1$. Now assume $d_{X_{j'}}(u_1)=1$. Then \( (V_1, \ldots, V_k) \) is \( (u_2, X_{j'}, u_5, X_{i'}) \)-decreasing if \( u_1u_5\notin E(G) \) or \( (\{u_1, u_3\}, X_{j'}, u_5, X_{i'}) \)-decreasing if \( u_1u_5\in E(G) \).
\end{enumerate}
Thus, $e(X_i,X_j)\ge 6$ for every $i\ne j \in [x]$.
\end{proof}

\section{Proof of Lemma \ref{l5}} \label{appendix B}
\begin{proof}
We begin with (i). A simple computation yields
\[{\psi}'(x)=\frac{1}{2}-\frac{1}{4}\cdot {{\left(2(\frac{8}{9}m-\beta_{1} +d_0-x)+\frac{1}{4}\right)}^{-1/2}} \]
and
\[{\psi}''(x)=-\frac{1}{4}\cdot {{\left(2(\frac{8}{9}m-\beta_{1} +d_0-x)+\frac{1}{4}\right)}^{-3/2}}<0,\]
which implies $\psi(x,c)$ is concave in $x$ for any fixed $c$.

For (ii), assume for contradiction that $\psi(8m/9-\beta_{1} +d,c)<0$. We obtain that
\begin{equation} \label{l5e1}
    \beta_{1} >\frac{2}{9}m+d-\frac{1}{3}\sqrt{2m+\frac{1}{4}}+\frac{1}{6}+c\geq \frac{2}{9}m+d-\frac{1}{3}\sqrt{2m+\frac{1}{4}}+\frac{1}{6}.
\end{equation}
Let $r_2(m)=-2m/3+(5\sqrt{2m+1/4})/3$. Note that $r_2(m)<0$ when $m\ge 18$.
By Lemma \ref{l4} (iii) and (\ref{l5e1}), we have
\[
   m'<-\frac{2}{3}m-3d+\frac{5}{3}\sqrt{2m+\frac{1}{4}}-\frac{5}{6}<0,
\]
a contradiction. It follows that $\psi(8m/9-\beta_{1} +d,c)\ge 0$.

Next, we prove (iii). Let $a:=4m/9+4\beta_{1}-d$. Substituting $a$ into $\psi(x,c)$, we have
\[
 \psi(a,c) =\frac{1}{4}\sqrt{\frac{8}{9}m-10\beta_{1} +4d+\frac{1}{4}}+\frac{3}{2}\beta_{1} -\frac{1}{3}\sqrt{2m+\frac{1}{4}}+\frac{1}{24}+c. 
\]
Note that $\beta_1$ must satisfy $8m/9-10\beta_{1} +4d+1/4 \ge 0$, from which we deduce
\[
\beta_{1} \le \gamma:=\frac{4}{45}m + \frac{2}{5}d + \frac{1}{40}.
\]
Let $\theta_{0}:=h(m)/9$ and $\theta_{1}:=h(m)/9+1/2$. Consider $\psi(a,c)$ as a function as $\beta_{1}$, let $g(\beta_{1}):=\psi(a,c)$. Its second derivative is \[
g''(\beta_{1})=-\frac{25}{4\left(\dfrac{8}{9}m - 10\beta_1 + 4d+ \dfrac{1}{4}\right)^{\frac{3}{2}}}<0.
\]
By the concavity of \(g(\beta_{1})\), it suffices to verify \(g(\theta_1) \geq 0\) and \(g(\gamma) \geq 0\) when \(\beta_1 \geq \theta_1\).
Substituting $\gamma$ into $g(\beta_{1})$, we have
\begin{align}
    g(\gamma) =&\frac{2}{15}m+\frac{3}{5}d-\frac{1}{3}\sqrt{2m+\frac{1}{4}}+\frac{3}{80}+\frac{1}{24}+c
    \geq  \frac{2}{15}\left(m-\frac{5}{2}\sqrt{2m+\frac{1}{4}}\right)+\frac{3}{5}+\frac{3}{80}+\frac{1}{24}\notag \\
    >& \frac{2}{15}\left(m-\frac{5}{2}\sqrt{2m+\frac{1}{4}}\right)
			 = -\frac{1}{5}r_2(m)
              > 0. \notag
     \end{align}
When $m\ge 18$, it is easy to show that 
\[
\sqrt{2m-\frac{5}{2}\sqrt{2m+\frac{1}{4}}-\frac{7}{16}}>\frac{2}{3}\sqrt{2m+\frac{1}{4}},
\quad 
\sqrt{2m-\frac{5}{2}\sqrt{2m+\frac{1}{4}}}-\sqrt{2m+\frac{1}{4}}\geq-\frac{3}{2}
\]
and
\[
\sqrt{2m - \frac{5}{2}\sqrt{2m+\frac{1}{4}} + \frac{5}{4} +9+ \frac{9}{16}} -\sqrt{2m+\frac{1}{4}}>-\frac{5}{4}.
\]
If $\beta_{1}\ge h(m)/9+1/2$, then
\begin{align}
g(\theta_{1})
    =&\frac{1}{4}\sqrt{\frac{8}{9}m-\frac{10}{9}h(m)-5 +4d+\frac{1}{4}}+\frac{3}{2}(\frac{1}{9}h(m)+\frac{1}{2}) -\frac{1}{3}\sqrt{2m+\frac{1}{4}}+\frac{1}{24}+c \notag \\ 
     \ge &\frac{1}{6}\left(\sqrt{2m-\frac{5}{2}\sqrt{2m+\frac{1}{4}}-\frac{7}{16}}-\sqrt{2m+\frac{1}{4}}\right) +\frac{17}{24} \notag \\
     =&\frac{1}{6}\left(\frac{-\frac{5}{2}\sqrt{2m+\frac{1}{4}}-\frac{7}{16}}{\sqrt{2m-\frac{5}{2}\sqrt{2m+\frac{1}{4}}-\frac{7}{16}}+\sqrt{2m+\frac{1}{4}}}\right) +\frac{17}{24}\notag \\
     >& -\frac{1}{6}\left(\frac{\frac{5}{2}\sqrt{2m+\frac{1}{4}}+\frac{7}{16}}{\frac{5}{3}\sqrt{2m+\frac{1}{4}}}\right) +\frac{17}{24}
     >-\frac{1}{6}(\frac{3}{2}+\frac{21}{80\sqrt{36+\frac{1}{4}}})+\frac{17}{24} 
     > 0.\notag
     \end{align}
     
For (iv), we have
\[
\psi(a+1,c)=\frac{1}{4}\sqrt{\frac{8}{9}m-10\beta_{1} +4d-\frac{7}{4}}+\frac{3}{2}\beta_{1} -\frac{1}{3}\sqrt{2m+\frac{1}{4}}+\frac{13}{24}+c.
\]
    
Note that $\beta_1$ must satisfy $8m/9-10\beta_{1} +4d-7/4 \ge 0$, from which we deduce that
\[
\beta_{1} \le \gamma_1:=\frac{4}{45}m + \frac{2}{5}d - \frac{7}{40}.
\]
Consider $\psi(a+1,c)$ as a function as $\beta_{1}$ and let $g_1(\beta_{1}):=\psi(a+1,c)$. Similarly, one can show $g_1(\beta_{1})$ is concave in $\beta_1$. Since $d\geq 1$, we have
\begin{align}
    g_1(\gamma_1) =&\frac{2}{15}m+\frac{3}{5}d-\frac{1}{3}\sqrt{2m+\frac{1}{4}}-\frac{21}{80}+\frac{13}{24}+c
    \geq \frac{2}{15}\left(m-\frac{5}{2}\sqrt{2m+\frac{1}{4}}\right)+\frac{3}{5}-\frac{21}{80}+\frac{13}{24}\notag \\
    >& \frac{2}{15}\left(m-\frac{5}{2}\sqrt{2m+\frac{1}{4}}\right)
			 = -\frac{1}{5}r_2(m)
              > 0 \notag
     \end{align}
     and
\begin{align}
   g_1(\theta_0) =& \frac{1}{4}\sqrt{\frac{8}{9}m-\frac{10}{9}h(m) +4d-\frac{7}{4}}+\frac{1}{6}h(m) -\frac{1}{3}\sqrt{2m+\frac{1}{4}}+\frac{13}{24}+c \notag \\
   >& \frac{1}{6}\left(\sqrt{2m-\frac{5}{2}\sqrt{2m+\frac{1}{4}}}-\sqrt{2m+\frac{1}{4}}\right)+\frac{11}{24} 
   \geq \frac{1}{6}\cdot(-\frac{3}{2})+\frac{11}{24} 
   >0. \notag
\end{align}
Thus, $\psi(a+1,c)\geq0$.

Finally we prove (v). Since $c\geq 1/4$, we have
\begin{align}
  g\left(\theta_0\right)
\geq &\frac{1}{4}\sqrt{\frac{8}{9}m - \frac{10}{9}h(m) + 4d + \frac{1}{4}}+\frac{1}{6}h(m) -\frac{1}{3}\sqrt{2m+\frac{1}{4}}+\frac{1}{24}+\frac{1}{4} \notag \\
\geq &\frac{1}{6}\sqrt{2m - \frac{5}{2}\sqrt{2m+\frac{1}{4}} + \frac{5}{4} +9+ \frac{9}{16}} -\frac{1}{6}\sqrt{2m+\frac{1}{4}}-\frac{1}{12}+\frac{1}{24}+\frac{1}{4} \notag \\
\geq &-\frac{1}{6}\cdot \frac{5}{4}+\frac{5}{24} =0.\notag
\end{align}    
This completes our proof.
\end{proof}
\end{appendices}

\end{document}